\documentclass[a4paper,11pt]{article}
 
\usepackage{graphicx}

\usepackage{amssymb}
\usepackage{amsthm}

\usepackage{nccmath}
\usepackage[mathscr]{euscript}
\usepackage{tikz}
\usepackage{dsfont}
\usepackage{color}
\usepackage{subfig}
\usepackage{xcolor}

\newcommand{\card}{\mathrm{card}}

\newcommand{\zero}{\mathbf{0}}
\newcommand{\R}{\mathbb{R}}

\newcommand{\co}{\mathrm{co}}

\newcommand{\A}{\mathbf{A}}

\newcommand{\x}{\mathbf{x}} 
\newcommand{\y}{\mathbf{y}} 
\newcommand{\z}{\mathbf{z}} 

\newtheorem{theorem}{Theorem}[section]

\newtheorem{lemma}{Lemma}[section]

\newtheorem{definition}{Definition}[section]

\begin{document}

  \title{A generalisation of  de la Vall\'{e}e-Poussin procedure to multivariate approximations}

 \author{
%
Nadezda Sukhorukova, \\
Swinburne University of Technology, John St, Hawthorn VIC 3122,\\
Australia and Federation University Australia,\\ 
Postal address: PO Box 663. Ballarat VIC 3353\\ 
 {nsukhorukova@swin.edu.au}
\and Julien Ugon, Federation University Australia,\\ 
Postal address: PO Box 663. Ballarat VIC 3353\\
 {j.ugon@federation.edu.au}}

%
%
\maketitle

\abstract{The theory of Chebyshev approximation has been extensively studied. In most cases, the optimality conditions are based on the notion of alternance or alternating sequence (that is, maximal deviation points with alternating deviation signs). There are a number of approximation methods for polynomial and polynomial spline approximation. Some of them are based on  the classical de la  Vall\'{e}e-Poussin procedure. In this paper  we demonstrate that under certain assumptions the classical de la Vall\'{e}e-Poussin procedure, developed for univariate polynomial approximation, can be extended to the case of multivariate approximation.  The corresponding basis functions are not restricted to be monomials.
}

{\bf Keywords:} {Multivariate polynomial, Chebyshev approximation, de la  Vall\'{e}e-Poussin procedure}

{\bf Subclass:} {41A10 \and  41A50 \and 41N10}

\section{Introduction}\label{sec:introduction}

The theory of Chebyshev approximation for univariate functions was developed in the late nineteenth (Chebyshev) and twentieth century (just to name a few \cite{nurnberger,rice67,Schumaker68}). Many papers are dedicated to polynomial and polynomial spline approximations, however, other types of  functions (for example, trigonometric polynomials) have also been used. In most cases, the optimality conditions are based on the notion of alternance (that is, maximal deviation points with alternating deviation signs). 

There have been several attempts to extend this theory to the case of multivariate functions. One of them is \cite{rice63}. The main obstacle in extending these results to the case of multivariate functions is that it is not very easy to extend the notion of monotonicity to the case of several variables.

The main contribution of this paper is the extention of  the classical de la Vall\'{e}e-Poussin procedure (originally developed for univariate polynomial approximation \cite{valleepoussin:1911}) to the case of multivariate approximation under certain assumptions. The corresponding basis functions are not restricted to be monomials (that is, non-polynomial approximation). 

The paper is organised as follows. In section~\ref{sec:convexObjective} we demonstrate that the corresponding optimisation problems are convex. Then, in section~\ref{sec:VPprocedure} we extend the classical de la Vall\'{e}e-Poussin procedure to the case of multivariate approximation. Finally, section~\ref{sec:conclusion} highlights our future research directions.

\section{Convexity of the objective function}\label{sec:convexObjective}
 Let us now formulate the objective function. Suppose that a continuous function $f(\x)$ is to be approximated by a function
 \begin{equation}\label{eq:model_function}
 L(\A,\x)=a_0+\sum_{i=1}^{n}a_ig_i(\x),
 \end{equation}
 where $L(\A,\x)$ is a modelling function, $g_i(\x),~i=1,\dots,n$ are the basis functions and the multipliers $\A = (a_0,a_1,\dots,a_n)$ are the corresponding coefficients. In the case of polynomial approximation, basis functions are monomials. In this paper, however, we do not restrict ourselves to polynomials. At a point \(\x\) the deviation between the function \(f\) (also referred as approximation function) and the approximation is:
 \begin{equation}
   d(\A,\x) = |f(\x) - L(\A,\x)|.
\end{equation}
   \label{eq:deviation}
 Then we can define the uniform approximation error over the set \(Q\) by
 \begin{equation}
   \label{eq:uniformdeviation}
\Psi(\A)=\sup_{\x\in Q} \max\{f(\x)-a_0-\sum_{i=1}^{n}a_ig_i(\x),a_0+\sum_{i=1}^{n}a_ig_i(\x)-f(\x)\}.
\end{equation}
 The approximation problem is
 \begin{equation}\label{eq:obj_fun_con}
   \mathrm{minimise~}\Psi(\A) \mathrm{~subject~to~} \A\in 
   \R^{n+1}.
 \end{equation}

 Since the function \(L(\A,\x)\) is linear in \(\A\), the approximation error function \(\Psi(\A)\), as the supremum of affine functions, is convex. Furthermore, its subdifferential at a point~\(\A\) is trivially obtained using the gradients of the active affine functions in the supremum (see \cite{Zalinescu2002} for details):
 \begin{equation}
   \label{eq:subdifferentialObjective}
   \partial \Psi(\A) = \co\left\{ \begin{pmatrix}
1\\
g_1(\x)\\
g_2(\x)\\
\vdots \\
g_n(\x)
\end{pmatrix}: \x \in E_+(\A),-\begin{pmatrix}
1\\
g_1(\x)\\
g_2(\x)\\
\vdots \\
g_n(\x)
\end{pmatrix}: \x\in E_-(\A)\right\},
\end{equation}
 where \(E_+(\A)\) and \(E_-(\A)\) are respectively the points of maximal positive and negative deviation (extreme points):
\begin{align*}
   E^+(\A) &= \Big\{\x\in Q:  f(\x)-L(\A,\x) = \max_{\y\in Q} d(A,\y)\Big\},\\
   E_- (\A)&= \Big\{\x\in Q: -f(\x)+ L(\A,\x) = \max_{\y\in Q} d(\A,\y)\Big\}.
 \end{align*}
Note that in the case of multivariate polynomial approximation, $g_i(\x)$, $i=1,\dots,n$ are monomials.

 Define by \(G^+\) and \(G^-\) the sets
 \begin{align*}
   G^+(\A) &= \Big\{(1,g_1(\x),\dots,g_n(\x))^T: \x\in E^+(\A)\Big\}\\
   G^-(\A) &= \Big\{(1,g_1(\x),\dots,g_n(\x))^T: \x\in E^-(\A)\Big\}
 \end{align*}
Assume that \(\card(E_+) + \card(E_-) = n+2\). 

The following theorem holds. We present the proof for completeness. 

\begin{theorem}\label{thm:main}(\cite{matrix})
$\A^*$ is an optimal solution to problem~(\ref{eq:obj_fun_con}) if and only if the convex hulls of the sets \(G^+(\A^*)\) and \(G^-(\A^*)\) intersect.
\end{theorem}

\begin{proof} 
The vector \(\A^*\) is an optimal solution to the convex problem \eqref{eq:obj_fun_con}  if and only if
\[
  \zero_{n+1} \in \partial \Psi(\A^*),
\]
where $\Psi$ is defined in \eqref{eq:uniformdeviation}.
Note that due to Carath\'eodory's theorem, $\zero_{n+1}$ can be constructed as a convex combination of a finite number of points (one more than the dimension of the corresponding space). Since the dimension of the corresponding space is $n+1$, it can be done using at most $n+2$ points.

Assume that in this collection of $n+2$ points $k$ points ($h_i,~i=1,\dots,k$) are from~$G^+(\A^*)$ and $n+2-k$ ($h_i,~i=k+1,\dots,n+2$) points are from $G^-(\A^*)$. Note that $0<k<n+2$, since the first coordinate is either~1 or $-1$ and therefore $\zero_{n+1}$ can only be formed by using both sets ($G^+(\A^*)$ and $-G^-(\A^*)$). Then
$$\zero_{n+1}=\sum_{i=1}^{n+2}\alpha_ih_i,~0\leq\alpha\leq 1.$$
Let $0<\gamma=\sum_{i=1}^{k}\alpha_i$, then
$$\zero_{n+1}=\sum_{i=1}^{n+2}\alpha_ih_i=\gamma\sum_{i=1}^{k}\frac{\alpha_i}{\gamma}h_i+(1-\gamma)\sum_{i=k+1}^{n+2}\frac{\alpha_i}{1-\gamma}h_i=\gamma h^+ +(1-\gamma)h^-,$$
where $h^+\in G^+(\A^*)$ and $h^-\in -G^-(\A^*)$. Therefore, it is enough to demonstrate that $\zero_{n+1}$ is a convex combination of two vectors, one from $G^+(\A^*)$ and one from $-G^-(\A^*)$.

By the formulation of the subdifferential of \(\Psi\) given by \eqref{eq:subdifferentialObjective}, there exists a nonnegative  number \(\gamma \leq 1\) and two vectors
\[
  g^+ \in \co\left\{ \begin{pmatrix}
1\\
g_1(\x)\\
g_2(\x)\\
\vdots \\
g_n(\x)
\end{pmatrix}: \x \in E^+(\A^*)\right\}, \]

and
\[
  g^- \in \co\left\{ \begin{pmatrix}
1\\
g_1(\x)\\
g_2(\x)\\
\vdots \\
g_n(\x)
\end{pmatrix}: \x \in E^-(\A^*)\right\}
\]
such that \(\zero = \gamma g^+ - (1-\gamma) g^-\). Noticing that the first coordinates \(g^+_1 = g^-_1 = 1\), we see that \(\gamma = \frac{1}{2}\). This means that \(g^+ - g^- = 0\). This happens if and only if
 \begin{equation}\label{eq:opt_main2}
 \co\left\{
\left(
\begin{matrix}
1\\
g_1(\x)\\
g_2(\x)\\
\vdots
\\
g_n(\x)\\
\end{matrix}
\right): \x \in E^+(\A^*)
 \right
  \}\cap
  \co\left\{
\left(
\begin{matrix}
1\\
g_1(\x)\\
g_2(\x)\\
\vdots
\\
g_n(\x)\\
\end{matrix}
\right): \x \in E^-(\A^*)
 \right \}\ne\emptyset.
 \end{equation}
 As noted before, the first coordinates of all these vectors are the same, and therefore the theorem is true, since if $\gamma$ exceeds one, the solution where all the components are divided by $\gamma$ can be taken as the corresponding coefficients in the convex combination.
 \end{proof}

\section{de la Vall\'{e}e-Poussin procedure for nonsingular basis}\label{sec:VPprocedure}
\subsection{Definitions and existing results}

We start with necessary definitions from convex analysis.

\begin{definition}
The relative interior of a set $S$ (denoted by $\textrm{relint} (S)$) is defined as its interior within the affine hull of $S$. 
That is,
$$
\textrm{relint}(S)= \{\textbf{x} \in S : \exists \varepsilon>0, B_\varepsilon(x)\cap \textrm{aff}(S)\subseteq S\},$$
where $B_\varepsilon(x)$ is a ball of radius $\varepsilon$ centred in $x$ and   $\textrm{aff}(S)$ is the affine hull of $S$. 
\end{definition}

A useful property of relative interiors of convex hulls of finite number of points is formulated in the following lemma. 
\begin{lemma}
Any relative interior point of a convex combination of a finite number of points can be presented as a convex combination of all these points with strictly positive convex combination coefficients and vice versa. 
\end{lemma}

In univariate case polynomial approximation, basis is an arbitrary collection of $n+2$ points, where $n$ is the number of monomials. What do we call basis in multivariate case? Based on necessary and sufficient optimality conditions (Theorem~\ref{thm:main}) the convex hulls built over positive and negative maximal deviation points should intersect. Is it always possible to partition $n+2$ points in to two subsets in such a way that the corresponding convex hulls are intersecting. The answer to this question is ``yes'', if $n\geq d$.  The following theorem holds.  
\begin{theorem}(Radon \cite{Radon1921})
Any set of $d+2$ points in $\R^d$ can be partitioned into two disjoint sets whose convex hulls intersect. 
\end{theorem}

\begin{definition}
A point in the intersection of these convex hulls is called a Radon point of the set.
\end{definition}
In the rest of the paper we assume that  $n\geq d$.

It will be demonstrated that it is not possible to extend  de la Vall\'{e}e-Poussin procedure to multivariate approximations without imposing additional assumptions (non-singular basis). It may be possible that some (or all) of these assumptions can be removed if we restrict ourselves to a particular class of basis functions (for example, monomials). This research direction is out of scope of this paper.

\begin{definition}
Consider a set \(\mathcal{S}\) of \(n+2\) points partitioned into two sets, the sets \(\mathcal{Y}\) of points with positive deviation and \(\mathcal{Z}\) of points with negative deviation. These points are said to form a \emph{basis} if the convex hulls of \(\mathcal{Y}\) and \(\mathcal{Z}\) intersect. Furthermore, if the relative interiors of the convex hulls intersect and any $(n+1)$ point subset of this basis form an affine independent system then the basis is said to be \emph{non-singular}.
\end{definition} 

\subsection{de la Vall\'{e}e-Poussin procedure for multivariate approximations}
\subsubsection{Classical univariate procedure}
The classical univariate de la Vall\'{e}e-Poussin procedure contains three steps.
\begin{enumerate}
\item For any basis ($n+2$ points) there exists a unique polynomial, such that the absolute deviation at the basis points is the same and the deviation sign is alternating. This polynomial is also called Chebyshev interpolation polynomial. 
\item If there is a point (outside of the current basis), such that the absolute deviation at this point is higher than at the basis points then this point can be included in the basis by removing one of the current basis points and the deviation signs are deviating.  
\item The absolute deviation of the new Chebyshev interpolating polynomial is at least as high as the absolute deviation for the original basis.   
\end{enumerate}
In the rest of this section we extend the procedure for a non-singular basis. 
\subsubsection{Step one extension}

We start with constructing Chebyshev interpolation polynomials. The following theorem holds.
\begin{theorem}
Assume that a system of points $\mathbf{y}_i,~i=1,\dots,N_+$ and $\mathbf{z}_i,~i=1,\dots,N_-$ forms a non-singular basis. Then there exists a unique polynomial deviating from $f$ at the points $\mathbf{y}_i,~i=1,\dots,N_+$ and $\mathbf{z}_i,~i=1,\dots,N_-$ by the same value and the deviation signs are opposite for $\mathbf{y}_i$ and $\mathbf{z}_i$. 
\end{theorem}
\begin{proof}
Consider the following linear system:
\begin{equation}\label{eq:main_system}
\left(
\begin{tabular}{ccc}
1&{$g(\mathbf{y}_1)$}&1\\
1&{$g(\mathbf{y}_2)$}&1\\
\vdots & \vdots & \vdots\\
1&{$g(\mathbf{y}_{N_+})$}&1\\
1&{$g(\mathbf{z}_1)$}&-1\\
1&{$g(\mathbf{z}_2)$}&-1\\ 
\vdots & \vdots & \vdots\\
1&{$g(\mathbf{z}_{N_-})$}&-1\\
\end{tabular}
\right)\left(
\begin{tabular}{c}
$\mathbf{A}$\\
$\sigma$\\
\end{tabular}
\right)
=\left(\
\begin{tabular}{c}
$f(\mathbf{y}_1)$\\
$f(\mathbf{y}_2)$\\
\vdots\\
$f(\mathbf{y}_{N_+})$\\
$f(\mathbf{z}_1)$\\
$f(\mathbf{z}_2)$\\
\vdots\\
$f(\mathbf{z}_{N_-})$\\
\end{tabular}
\right),
\end{equation}
where $\A$ represents the parameters of the polynomial, while $\sigma$ is the deviation. If $\sigma=0$, there exists a polynomial passing through the chosen points (interpolation).   
Denote the system matrix in~(\ref{eq:main_system}) by $M$. Since the basis is non-singular, that is, the  relative interiors of sets ${\cal{Y}}$ and ${\cal{Z}}$ are intersecting, there exist two sets of strictly positive coefficients $$\alpha_1,\dots,\alpha_{N_+}:~ \sum_{i=1}^{N_+}\alpha_i=1$$
and
$$\beta_1,\dots,\alpha_{N_-}:~ \sum_{i=1}^{N_-}\beta_i=1,$$
such that 
\begin{equation}\label{eq:intersecting}
\sum_{i=1}^{N_+}\alpha_ig(\mathbf{y}_i)=\sum_{i=1}^{N_-}\beta_i g(\mathbf{z}_i).
\end{equation}  
Multiply the first row of $M$ by the convex coefficient $\alpha_1$ from~(\ref{eq:intersecting}). For each remaining row of $M$ one can apply the following update: 
\begin{itemize}
\item  multiply by the corresponding convex coefficient and add all the rows that correspond to the vertices with the same deviation sign as the first row;
\item multiply by the corresponding convex coefficient and subtract all the rows that correspond to the vertices with the deviation sign opposite to the sign of the first row. 
\end{itemize}
Then
\begin{equation}\label{eq:det_tilde_M}
\alpha_l\det(\tilde{M})=2(-1)^{l+2+i}\det(M^+_l),\;~l=1,\dots, N_{+}, 
\end{equation}
where $M^+_i$ is obtained from~$\tilde{M}$ by removing the last column and the $i-$th row and $M^-_j$ is obtained from~$\tilde{M}$ by removing the last column and the $(N_{+}+j)$-th row. Also note that
\begin{equation}
\det(M^+_i)=2(-1)^{l+2+N_{+}+j+1}\det(M^-_j),~l=1,\dots, N_{+}.
\end{equation}

If now we evaluate the the determinant of~$M$ directly, then 
\begin{equation}\label{eq:directly_det_M}
\det M=\sum_{i=1}^{N_+}(-1)^{l+2+i}\Delta_i+\sum_{j=N_{+}+1}^{N_{+}+N_{-}}(-1)^{l+2+j+1}\Delta_j.
\end{equation}
Based of~(\ref{eq:det_tilde_M}), each component in the right hand side of~(\ref{eq:directly_det_M}) has the same sign. 
Therefore, the linear system~(\ref{eq:main_system}) has a unique solution for any right hand side of the system.
\end{proof}

%
%
%
%
%


Note that the division into ``positive'' and ``negative'' basis points does not mean that the deviation sign is positive for ``positive'' basis points and negative for ``negative'' basis points. The actual deviation sign also depends on the sign of $\sigma$ from~(\ref{eq:main_system}). 
%

%

Extending the notion of Chebyshev interpolating polynomial to the case of multivariate approximation and not restricting ourselves to polynomials, define the following. 
\begin{definition}
A modelling function $L(\A,\x)$ from~(\ref{eq:model_function}) that deviates at the basis points by the same absolute value from  its approximation function and the deviation signs are opposite for any two points if they are selected from different basis subsets (positive or negative) is called Chebyshev interpolation modelling function.
\end{definition}

The additional requirement for a basis to be non-singular may be removed by
\begin{itemize}
\item restricting to some particular types of basis functions (for example, polynomials);
\item allowing the system~(\ref{eq:main_system}) to have more than one solution.
\end{itemize} 
These will be included in our future research directions.
\subsubsection{Step two extension}

Our next step is to demonstrate 
\begin{theorem}
Consider two intersecting sets \(\mathcal{Y}\) and~\(\mathcal{Z}\) such that the points in \(\mathcal{Y}\) all have the same deviation and opposite deviation to all the points in \(\mathcal{Z}\) (\(g(\tilde{y}) = -g(\tilde{z}), \forall \tilde{y}\in \mathcal{Y}, \tilde{z}\in\mathcal{Z}\)). Assume now that $g(\y) = g(\tilde{y}), \forall \tilde{y} \in \mathcal{Y}$, and that the set 
$$\mathcal{K}=\textrm{relint}(\{{\cal{Y}}\cup g(\y)\})\cap\textrm{relint}({\cal{Z}}) \neq \emptyset.$$ There exists a point in the combined collection of vertices of~${\cal{Y}}$ and ${\cal{Z}}$, that can be removed while~$\y$ is included in~${\cal{Y}}$, such that the updated sets~${\cal{\tilde{Y}}}$ and ${\cal{\tilde{Z}}}$ intersect. 
\end{theorem}
\begin{proof}
Since ${\textrm{relint}}({\cal{Y}})\cap\text{relint}({\cal{Z}})\ne \emptyset$, there exist strictly positive coefficients 
$$\alpha_i,~ i=1,\dots,N_+$$ and  $$\beta_i,~j=1,\dots,N_-,$$ such that $\sum_{i=1}^{N_+}\alpha_i=1$ and $\sum_{j=1}^{N_-}\beta_j=1$.

Since 
$\mathcal{K}\ne\emptyset$ there exist strictly positive coefficients $$\alpha,~\tilde{\alpha}_i,~i=1,\dots,N_+$$ such that $\alpha+\sum_{i=1}^{N_+}\tilde{\alpha}_i=1$ and  $\tilde{\beta}_i$, $j=1,\dots,N_-$, such that $$\sum_{j=1}^{N_-}\tilde{\beta}_j=1.$$ 



Find
\begin{equation}\label{eq:gamma}
\gamma=\min\left\{\min_{i=1,\dots,N_+}{\tilde{\alpha}_i\over \alpha_i},\min_{j=1,\dots,N_-}{\tilde{\beta}_j\over \beta_j}\right\}.
\end{equation}

First, assume that $\gamma={\tilde{\alpha}_1\over\alpha_1}$. Note that $\alpha_1\ne 0$,
then~(\ref{eq:intersecting}) can be written as
$$\y_1={1\over\alpha_1}\left(\sum_{j=1}^{N_-}\beta_jg(\z_j)-\sum_{i=2}^{N_+}\alpha_i g(\y_i)\right).$$
Then, the convex hull with the new point $\y$ is
\begin{equation}
\alpha g(\y)+{\tilde{\alpha}_1\over\alpha_1}\left(\sum_{j=1}^{N_-}\beta_jg(\z_j)-\sum_{i=2}^{N_+}\alpha_ig(\y_i)\right)+\sum_{i=2}^{N_+}\tilde{\alpha_i g(\y_i)}=\sum_{j=1}^{N_-}\tilde{\beta}_jg(\z_j)
\end{equation}
and finally
\begin{equation}
\alpha g(\y)+\sum_{i=2}^{N_+}(\tilde{\alpha}_i-{\tilde{\alpha}_i\over \alpha_i})g(\y_i)=\sum_{j=1}^{N_-}(\tilde{\beta}_j-{\tilde{\alpha}_1\over \alpha_1})g(\z_j).
\end{equation}
Since $\alpha_i>0$, $i=1,\dots,N_+$ and the definition of $\gamma$, one can obtain that for any $i=1,\dots,N_+$
\begin{equation}
\tilde{\alpha}_i-{\tilde{\alpha}_1\over \alpha_1}\geq \tilde{\alpha}_i-{\tilde{\alpha}_i\over \alpha_i}=0.
\end{equation}
Similarly, for any $j=1,\dots,N_-$, $$\tilde{\beta}_j-{\tilde{\alpha}_1\over\alpha_1}\beta_j\geq 0.$$
Note that
\begin{equation}
\sum_{j=1}^{N_-}(\tilde{\beta}_j-{\tilde{\alpha}_1\over \alpha_1}\beta_j)=1-{\tilde{\alpha}_1\over \alpha_1}
\end{equation}
and 
\begin{equation}
\alpha+\sum_{i=2}^{N_+}\tilde{\alpha}_i-{\tilde{\alpha}_1\over\alpha_1}\sum_{i=2}^{N_+}\alpha_i=\alpha+(1-\alpha\tilde{\alpha}_1)-{\tilde{\alpha}_1\over\alpha_1}=1-{\tilde{\alpha}_1\over\alpha_1}=1-\gamma.
\end{equation}
Since $\alpha$ is strictly positive, $\gamma<1$. Therefore, the new point can be included instead of $\y_1$ and the convex hulls of the updated sets are intersecting (and so their relevant interiors).

Second, assume that $\gamma={\tilde{\beta}_1\over\beta_1}$. Note that $\beta_1\ne 0$, otherwise $\y$ can be included instead of $\z_1$.

Similarly to part~1, obtain
\begin{equation}
\alpha g(\y)+\sum_{i=1}^{N_+}(\tilde{\alpha}_i-{\tilde{\beta}_1\over\beta_1}\alpha_i)g(\y_i)=\sum_{j=2}^{N_-}(\tilde{\beta}_j-{\tilde{\beta}_1\over\beta_1}\beta_j)g(\z_j).
\end{equation} 
Since
$$\alpha+1-\alpha-{\tilde{\beta}_1\over\beta_1}=1-\tilde{\beta}_1-{\tilde{\beta}_1\over\beta_1}(1-\beta_1)=1-{\tilde{\beta}_1\over\beta_1}>0,$$
the convex hulls of the updated sets are intersecting.
\end{proof}

Note that for the extension of this step we only need the assumption that the relative interiors are intersecting, moreover, if this is the case, the new basis preserves this property.

\subsubsection{Step three extension}

The final step is to show that the proposed exchange rule leads to a modelling function whose deviation at the new basis is strictly higher than the deviation at the points of the original basis.

\begin{theorem}
Assume that a point with a higher absolute deviation is included in the basis instead of one of the points of the original basis (which is also non-singular). The absolute deviation of the Chebyshev interpolation modelling function that corresponds to the new basis is higher than the one of the Chebyshev interpolation modelling function on the original basis.
\end{theorem}
\begin{proof}
Denote by \[\mathcal{Y} = \{\y_i,~i=1,\dots,N_+\}\] and \[\mathcal{Z} = \{\z_j,~j=1,\dots,N_-\}\] respectively. Assume that \(\tilde{\mathcal{Y}} = \mathcal{Y}\cup \{y\}\setminus \{y_1\}\) and \(\tilde{Z} = \mathcal{Z}\) (when the a point from the set \(\mathcal{Z}\) is removed instead, the proof is similar.)

Since the convex hulls of positive and negative deviation points are intersecting,   there exist nonnegative convex coefficients 
\begin{itemize}
\item $\alpha_1,\dots,\alpha_{N_+}: \sum_{i=1}^{N_+}\alpha_i=1$ and $\beta_1,\dots,\beta_{N_-}: \sum_{j=1}^{N_-}\beta_j=1$ (original basis);
\item $\alpha,~\tilde{\alpha}_2,\dots,\tilde{\alpha}_{N_+}: \alpha+\sum_{i=2}^{N_+}\tilde{\alpha}_i=1$ and $\beta_1,\dots,\beta_{N_-}: \sum_{j=1}^{N_-}\beta_j=1$ (new basis), 
\end{itemize}
such that on the original basis
\begin{equation}\label{eq:convex_hulls_new_basis}
\sum_{i=1}^{N_+}\alpha_i\y_i-\sum_{j=1}^{N_-}\beta_j\z_j=\mathbf{0}
\end{equation}
and on the new basis
\begin{equation}\label{eq:convex_hulls_original_basis}
\alpha\y+\sum_{i=2}^{N_+}\tilde{\alpha}_i\y_i-\sum_{j=1}^{N_-}\tilde{\beta}_j\z_j=\mathbf{0}
\end{equation}
Systems~(\ref{eq:convex_hulls_new_basis}) is equivalent to 
\begin{equation}
\left[\alpha,\tilde{\alpha}_2,\dots,\tilde{\alpha}_{N_+},\tilde{\beta}_1,\dots,\tilde{\beta}_{N_-}\right]\left[\begin{matrix}
\y\\
\y_2\\
\vdots\\
\y_{N_+}\\
\z_1\\
\vdots\\
\z_{N_-}
\end{matrix}
\right]=\mathbf{0}.
\end{equation}

Then  
\begin{equation}
\left[\alpha,\tilde{\alpha}_2,\dots,\tilde{\alpha}_{N_+},\tilde{\beta}_1,\dots,\tilde{\beta}_{N_-}\right]\left[\begin{matrix}
1&\y\\
1&\y_2\\
\vdots&\vdots\\
1&\y_{N_+}\\
1&\z_1\\
\vdots&\vdots\\
1&\z_{N_-}
\end{matrix}
\right]\A=\mathbf{0}
\end{equation}
for any $\A\in\R^{n+1}$.
Let $\A_{o}$ and $\A_{new}$ be parameter coefficients of the Chebyshev interpolation modelling functions that correspond to the original and new basis respectively. Then 
\begin{equation}\label{eq:orig_cheb_inter_pol}
\alpha P_n(\A_{o},\y)+\sum_{i=2}^{N_+}\tilde{\alpha}_iP_n(\A_{o},\y_i)-\sum_{j=1}^{N_-}\tilde{\beta}_jP_n(\A_{o},\z_j)=0
\end{equation} 
and 
\begin{equation}\label{eq:new_cheb_inter_pol}
\alpha P_n(\A_{new},\y)+\sum_{i=2}^{N_+}\tilde{\alpha}_iP_n(\A_{new},\y_i)-\sum_{j=1}^{N_-}\tilde{\beta}_jP_n(\A_{new},\z_j)=0.
\end{equation} 
Assume that 
\begin{equation}
f(\y_1)-P_n(\A_{new},\y_1)=\sigma_{new}>0.
\end{equation}
Then 
\begin{equation}
\sigma_{new}+P_n(\A_{new},\y)=f(\y),
\end{equation}
\begin{equation}
\sigma_{new}+P_n(\A_{new},\y_i)=f(\y_i),~i=2,\dots,N_+,
\end{equation}
and
\begin{equation}
-\sigma_{new}+P_n(\A_{new},\z_j)=f(\z_j),~j=2,\dots,N_-.
\end{equation}
Due to~(\ref{eq:orig_cheb_inter_pol})-(\ref{eq:new_cheb_inter_pol})
\begin{align*}
2\sigma_{new}=&\\
=&\alpha(f(\y)-P_n(\A_{o},\y)+\sum_{i=2}^{N_+}\tilde{\alpha_i}(f(\y_i)-P_n(\A_{o},\y_i))-\sum_{j=1}^{N_-}\tilde{\beta_j}(f(\z_j)-P_n(\A_{o},\z_i))\\
  >&2\sigma_{o}.\\
  \end{align*}
  Therefore, $\sigma_{new}>\sigma_{o}.$
\end{proof}

Therefore, the notion of basis and de la Vall\'{e}e-Poussin procedure is extended to multidimensional functions. Also, it has been extended to any basis functions (not only traditional polynomials). If the newly obtained basis is non-singular, one can make another de la Vall\'{e}e-Poussin procedure step.

\section{Further research directions}\label{sec:conclusion}

We will extend the results to the case when the basis is singular. In order to do this, we need to remove two assumptions.
\begin{enumerate}
\item Any $(n+1)$ point subset of the basis ($n+2$ points) form an affine independent system.
\item Relative interiors of the convex hulls of positive and negative maximal deviation points (restricted to basis) are intersecting. 
\end{enumerate}

The first assumption may not be removed for an arbitrary type of basis function. However, it may be possible to remove this assumption for some special types of functions (for example, polynomials). The removal of the second assumption may lead to dimension reduction.  These will be included in our future research directions.

\section{Acknowledgements}

This paper was inspired by the discussions during a recent
MATRIX program ``Approximation and Optimisation\rq{}\rq{}  that took place in July 2016. We are thankful to
the MATRIX organisers,  support team and participants for a terrific research atmosphere and productive discussions.

%
%

    \bibliographystyle{amsplain}


\end{document}